\documentclass[11pt]{article}
\usepackage{amsmath,amssymb,amsthm}
\hfuzz=\maxdimen
\newtheorem{thm}{Theorem}[section]
\newtheorem{pro}[thm]{Proposition}

\newtheorem{lem}[thm]{Lemma}

\newtheorem{cor}[thm]{Corollary}

\theoremstyle{definition}
\newtheorem{exam$s$-union ple}[thm]{Example}

\newtheorem{rmk}[thm]{Remark}

\begin{document}
\date{}

\title{On cover-free families of finite vector spaces\footnote{Supported by  National Natural Science Foundation of China (12171028, 12371326).}}

\author{
 {\small Yunjing  Shan,}  {\small Junling  Zhou}\\
{\small School of Mathematics and Statistics}\\ {\small Beijing Jiaotong University}\\
  {\small Beijing  100044, China}\\
 {\small jlzhou@bjtu.edu.cn}\\
}

\maketitle

\begin{abstract}
  There is a large literature on cover-free families of finite sets, because of their many applications in combinatorial group testing, cryptographic and communications. This work studies the generalization of cover-free families from sets to finite vector spaces. Let $V$ be an $n$-dimensional vector space over the finite field $\mathbb{F}_{q}$ and let $\left[V\atop k\right]_q$ denote the family of all $k$-dimensional subspaces of $V$. A family $\mathcal{F}\subseteq \left[V\atop k\right]_q$ is called cover-free if there are no three distinct subspaces $F_{0}, F_{1}, F_{2}\in \mathcal{F}$ such that $F_{0}\leq (F_{0}\cap F_{1})+(F_{0}\cap F_{2})$.  A family $\mathcal{H}\subseteq \left[V\atop k\right]_q$ is called a $q$-Steiner system $S_{q}(t, k, n)$ if for every $T\in \left[V\atop t\right]_q$, there is exactly one $H\in \mathcal{H}$ such that $T\leq H$. In this paper we investigate cover-free families in the vector space $V$. Firstly, we determine the maximum size of a cover-free family in $\left[V\atop k\right]_q$. Secondly, we characterize the structures of all maximum cover-free families which are closely related to $q$-Steiner systems.
\end{abstract}


\section{Introduction}

{\it Cover-free} families of finite sets were considered from different subjects such as extremal set theory, information theory, combinatorial designs, codes and group testing \cite{e, group, f, s}. Cover-free families were also considered for many cryptographic
problems such as frameproof codes and traceability schemes \cite{a1, a2, fra1, fra}, key storage \cite{b1, b2} and multi-receiver authentication \cite{c3}.

Let $X$ be an $n$-element set and let $\tbinom{X}{k}$ denote the set of all $k$-element subsets of $X$. 
A family $\mathcal{F}\subseteq \tbinom{X}{k}$ is called {\it cover-free} if there are no three distinct sets $F_{0}, F_{1}, F_{2}\in \mathcal{F}$ such that $F_{0}\subseteq F_{1}\cup F_{2}$. It is obvious that $\mathcal{F}=\{X\}$ is the maximum cover-free family if $n=k$; $\mathcal{F}=\{F, F'\}$ is the maximum cover-free family where $F\cup F'=X$ and $F, F'\in\tbinom{X}{k}$ if $n=k+1$. A family $\mathcal{H} \subseteq \tbinom{X}{k}$ is called a {\it Steiner system} $S(t, k, n)$ on $X$ if for every $T\in \tbinom{X}{t}$, there is exactly one $H\in \mathcal{H}$ such that $T\subseteq H$. Obviously, we have $|\mathcal{H}|=\frac{\tbinom{n}{t}}{\tbinom{k}{t}}$.

Erd\H{o}s, Frankl and F\"{u}redi \cite{two} discussed cover-free  families of sets in detail and established the following theorem.

\begin{thm}\label{sc}{\rm(\cite[Theorem 1]{two})}
Let $k, t$ and $n$ be positive integers with $n>$$k+1$. Suppose $\mathcal{F}\subseteq \tbinom{X}{k}$ is a cover-free  family. Then the following hold.
\begin{itemize}
\item[\rm(i)] When $k=2t,$
\begin{equation*}
|\mathcal{F}| \leq \frac{\tbinom{n-1}{t}}{\tbinom{2t-1}{t}}.
\end{equation*}
Moreover, equality holds if and only if $\mathcal{F}=\{\{x\}\cup S: S\in \mathcal{S}\}$, where $x\in X$, $\mathcal{S}$ is a Steiner system $S(t, 2t-1, n-1)$ on $X\backslash \{x\}$.

\item[\rm(ii)] When $k=2t-1,$
\begin{equation*}
|\mathcal{F}| \leq \frac{\tbinom{n}{t}}{\tbinom{2t-1}{t}}.
\end{equation*}
Moreover, equality holds if and only if $\mathcal{F}$ is a Steiner system $S(t, 2t-1, n)$ on $X$.

\end{itemize}
\end{thm}


The problems in extremal set theory have natural extensions to families of subspaces over a finite field. This is the first work to generalize cover-free families from finite sets to finite vector spaces. Throughout the paper we always let $V$ be an $n$-dimensional vector space over the finite field $\mathbb{F}_{q}$. Let $\left[V\atop k\right]_q$ denote the family of all $k$-dimensional subspaces of $V$. For $m\in \mathbb{R}, k\in \mathbb{Z}^{+}$, define the Gaussian binomial coefficient by
$$\left[m\atop k\right]_{q}:=\prod\limits_{i=0}^{k-1}\dfrac{q^{m-i}-1}{q^{k-i}-1}.$$
Obviously, the size of $\left[V\atop k\right]_q$ is $\left[n\atop k\right]_{q}$. If $k$ and $q$ are fixed, then $\left[m\atop k\right]_{q}$ is a continuous function of $m$ which is positive and strictly increasing when $m\geq k$. If there is no ambiguity, the subscript $q$ can be omitted.

We denote $A\leq B$ if $A$ is a subspace of $B$. For any two subspaces $A$, $B\leq V$, let $A+B$ denote the linear span of $A\cup B$ and let $A\oplus B$ denote the direct sum of $A$ and $B$ with $A\cap B=\{\textup{\textbf{0}}\} $. In the absence of parentheses, we default to perform the $``\bigcap"$ operation first and then the $``+"$ operation. Let $\mathcal{F}\subseteq \left[V\atop k\right]$ be a family of subspaces, we say that $\mathcal{F}$ is {\it intersecting} if for all $F$, $F'\in\mathcal{F}$, we have ${\rm dim}$$(F\cap F')\geq 1$. A family $\mathcal{H} \subseteq \left[V\atop k\right]$ is called a {\it $q$-Steiner system} $S_{q}(t, k, n)$ on $V$ if for every $T\in \left[V\atop t\right]$, there is exactly one $H\in \mathcal{H}$ such that $T\leq H$. Obviously, we have $|\mathcal{H}|=\frac{\left[n\atop t\right]}{\left[k\atop t\right]}$.

While many results about sets have been generalized to vector spaces, not so much is known about their $q$-analogs because adapting combinatorial techniques to vector spaces can be challenging. In this paper, we will extend Theorem~\ref{sc} to vector spaces. We begin with the definition of a $q$-analog of a cover-free family.

Let $F_{0}, F_{1}$ and $F_{2}$ be three distinct subsets in a set $X$. Observe that $F_{0}\subseteq F_{1}\cup F_{2}$ is equivalent to $F_{0}\subseteq (F_{0}\cap F_{1})\cup (F_{0}\cap F_{2})$. So we define a family $\mathcal{F}\subseteq \left[V\atop k\right]$ to be {\it cover-free} if there are no three distinct subspaces $F_{0}, F_{1}, F_{2}\in \mathcal{F}$ such that $F_{0}\leq F_{0}\cap F_{1}+F_{0}\cap F_{2}$. It is obvious that if $n=k$, then $|\mathcal{F}|\leq 1$. We only need to consider $n\geq k+1$. Our main result is as follows.


\begin{thm}\label{vsc}
Let $k, t$ and $n$ be positive integers with $n\geq k+1$. Suppose $\mathcal{F}\subseteq \left[V\atop k\right]$ is a cover-free family. Then the following hold.
\begin{itemize}
\item[\rm(i)] When $k=2t,$
\begin{equation*}
|\mathcal{F}| \leq \frac{\left[n-1\atop t\right]}{\left[2t-1\atop t\right]}.
\end{equation*}
Moreover, equality holds if and only if $\mathcal{F}=\{E\oplus S: S\in \mathcal{S}\}$, where $E\in \left[V\atop 1\right]$ and $\mathcal{S}$ is a $q$-Steiner system $S_{q}(t, 2t-1, n-1)$ on some $W\in \left[V\atop n-1\right]$ satisfying $E\oplus W=V$.
\item[\rm(ii)] When $k=2t-1,$
\begin{equation*}
|\mathcal{F}| \leq \frac{\left[n\atop t\right]}{\left[2t-1\atop t\right]}.
\end{equation*}
Moreover, equality holds if and only if $\mathcal{F}$ is a $q$-Steiner system $S_{q}(t, 2t-1, n)$ on $V$.
\end{itemize}
\end{thm}

We may begin with a verification of the relationship between cover-free families of subspaces with odd dimension and those with even dimension. Meanwhile, we need to show that the structures defined by $q$-Steiner systems in Theorem~\ref{vsc} are indeed cover-free families. To assist with these tasks, we require some elementary results in vector spaces.

\begin{lem}\label{vv}
\begin{itemize}
   \item[{\rm(i)}] Suppose that $T\in \left[V\atop t\right]$, $R\in \left[V\atop r\right]$ and~$T\cap R=\{\textup{\textbf{0}}\}$. If  $T\leq F\leq V$, then $F\cap (T\oplus R)=T\oplus R\cap F$.
   \item[{\rm(ii)}] Suppose that  $T\in \left[V\atop t\right]$, $P\in \left[V\atop p\right]$ and $T\cap P=\{\textup{\textbf{0}}\}$. If $R\in \left[P\atop r\right]$ and $S\in \left[P\atop s\right]$, then $(T\oplus R)\cap (T\oplus S)=T\oplus R\cap S$.
\end{itemize}
\end{lem}
\begin{proof}~
With the assumption of (i), it is obvious that $T\oplus (R\cap F)\leq F\cap (T\oplus R)$. Let ${\rm \dim}(F)=k$. Since
\begin{equation*}
\begin{array}{rl}
{\dim}(F\cap (T\oplus R))\!\!\!&=k+(r+t)-{\dim}(F+R)\\[.3cm]
&=k+(r+t)-(k+r-{\dim}(R\cap F))\\[.3cm]
&=t+{\dim}(R\cap F)\\[.3cm]
&={\dim}(T\oplus R\cap F),
\end{array}
\end{equation*}
this proves (i).

Since $R\cap S\leq R\cap (T\oplus S)$ and
\begin{equation*}
\begin{array}{rl}
{\dim}(R\cap (T\oplus S))\!\!\!&=r+(t+s)-{\dim}(T+R+S)\\[.3cm]
&=r+(t+s)-(t+{\dim}(R+S)-0)\\[.3cm]
&={\dim}(R\cap S),
\end{array}
\end{equation*}
we have $R\cap (T\oplus S)=R\cap S$. Then using the condition of (ii) and taking $F=T\oplus S$ in (i) gives $(T\oplus R)\cap (T\oplus S)=T\oplus R\cap (T\oplus S)=T\oplus R\cap S$.
\end{proof}
\begin{pro}\label{p22}
For $\mathcal{F}\subseteq \left[V\atop k\right]$, define
\begin{equation}\label{a3}
\mathcal{F}'=\{E\oplus F: F\in \mathcal{F}\},
\end{equation}
where $E$ is a $1$-dimensional vector space outside of $\left[V\atop 1\right]$. Then the following hold.
\begin{itemize}
\item[\rm(i)] If $\mathcal{F}$ is a cover-free family in $\left[V\atop k\right]$, then $\mathcal{F}'$ is a cover-free family in $\left[E\oplus V\atop k+1\right]$.
\item[\rm(ii)] If $\mathcal{F}$ is a $q$-Steiner system $S_{q}(t, 2t-1, n)$, then $\mathcal{F}$ is a cover-free family in $\left[V\atop 2t-1\right]$ and $\mathcal{F}'$ is a cover-free family in $\left[E\oplus V\atop 2t\right]$.
\end{itemize}
\end{pro}
\begin{proof}~${\rm(i)}$ Suppose that $\mathcal{F}$ is a cover-free family in $\left[V\atop k\right]$. If there are three distinct subspaces $F_{0}, F_{1}, F_{2}\in \mathcal{F}\subseteq \left[V\atop k\right]$ such that
\begin{align}
E\oplus F_{0}\leq &(E\oplus F_{0})\cap (E\oplus F_{1})+(E\oplus F_{0})\cap (E\oplus F_{2})\notag\\
=&(E\oplus F_{0}\cap F_{1})+(E\oplus F_{0}\cap F_{2})\tag{{\rm by~Lemma~\ref{vv}(ii)}}\\
=&E\oplus F_{0}\cap F_{1}+F_{0}\cap F_{2}\notag,
\end{align}
then $F_{0}=V\cap(E\oplus F_{0})\leq V\cap(E\oplus F_{0}\cap F_{1}+F_{0}\cap F_{2})=F_{0}\cap F_{1}+F_{0}\cap F_{2}$ by Lemma \ref{vv}, which contradicts that $\mathcal{F}$ is cover-free. As a result $\mathcal{F}'$ is a cover-free  family in the $(n+1)$-dimensional vector space $E\oplus V$.

${\rm(ii)}$ Suppose that $\mathcal{F}$ is a $q$-Steiner system $S_{q}(t, 2t-1, n)$. Thus ${\dim}(A\cap A')\leq t-1$ holds for all $A, A'\in \mathcal{F}$ with $A\neq A'$. Hence, we have
$${\dim}(F\cap A+F\cap A')\leq 2t-2<{\dim}(F)$$
for any $F\in \mathcal{F}\backslash \{A, A'\}$, that is, $F\nleq (F\cap A)+(F\cap A')$. Then $\mathcal{F}$ is a cover-free family in $\left[V\atop 2t-1\right]$. Consequently, we have that $\mathcal{F}'$ is a cover-free family in $\left[E\oplus V\atop 2t\right]$ by (i).
\end{proof}

\begin{rmk}\label{rm}
By Proposition~\ref{p22}, an $S_{q}(t, 2t-1, n)$ on $V$ gives rise to a cover-free family of cardinality $\frac{\left[n\atop t\right]}{\left[2t-1\atop t\right]}$ in $\left[V\atop 2t-1\right]$; and an $S_{q}(t, 2t-1, n-1)$ on $W$ gives a cover-free family of cardinality $\frac{\left[n-1\atop t\right]}{\left[2t-1\atop t\right]}$ in $\left[V\atop 2t\right]$ where $V=E\oplus W$. So what remains to prove Theorem~\ref{vsc} is to handle the upper bound of $|\mathcal{F}|$ and prove the ``only if~\!'' part of the equality. Actually the proof can ba reduced to the case of even $k$ in part (i). Suppose that (i) has been proved and $\mathcal{F}$ is a cover-free family in $\left[V\atop 2t-1\right]$. Then $\mathcal{F}'$ defined in~\eqref{a3} is a cover-free family in $\left[E\oplus V\atop 2t\right]$ and hence $|\mathcal{F}|=|\mathcal{F}'|\leq \frac{\left[n\atop t\right]}{\left[2t-1\atop t\right]}$ by Proposition~\ref{p22}. Further if $|\mathcal{F}|=\frac{\left[n\atop t\right]}{\left[2t-1\atop t\right]}$, then $\mathcal{F}'$ is a cover-free family in $\left[E\oplus V\atop 2t\right]$ of size $\frac{\left[n\atop t\right]}{\left[2t-1\atop t\right]}$ and $\mathcal{F}$ is an $S_{q}(t, 2t-1, n)$ by Theorem~\ref{vsc}(i). This means that part~(ii) also holds.
\end{rmk}

\section{ Preliminaries}
In this section, we recall a number of basic results and establish several new lemmas in the vector spaces, which are essential for our proofs. For a family $\mathcal{H}\subseteq \left[V\atop k\right]$, we define its {\it shadow} by
\begin{equation*}
\bigtriangleup(\mathcal{H})=\left\{G\in \left[V\atop k-1\right]:G\leq H~{\rm for~some}~H\in \mathcal{H}\right\}.
\end{equation*}
Now, we introduce the celebrated Erd\H{o}s-Ko-Rado theorem for vector spaces.

\begin{thm}\label{EKR}{\rm(\cite[Theorem 1]{ekr vector},~\cite[Theorem 3]{ekrbc})}
Suppose $\mathcal{H}\subseteq \left[V\atop k\right]$ is an intersecting family. Then we have

\begin{equation*}
|\mathcal{H}|\!\leq\!
\begin{cases}
\left[n-1\atop k-1\right], &\text{if~$n\geq 2k$},\\[.3cm]
\left[2k-1\atop k\right], &\text{if~$2k-1\leq n\leq2k$}.
\end{cases}\
\end{equation*}
Moreover, equality holds if and only if one of the following holds:
\begin{itemize}
\item[\rm(i)]~$n>2k$ and $\mathcal{H}=\left\{H\in\left[V\atop k\right]:E\leq H\right\}$ for some $E\in\left[V\atop 1\right]$.
\item[\rm(ii)]~$n=2k$ and $\mathcal{H}=\left\{H\in\left[V\atop k\right]:E\leq H\right\}$ for some $E\in\left[V\atop 1\right]$ or $\mathcal{H}=\left[Y\atop k\right]$ for some $Y\in\left[V\atop 2k-1\right]$.
\end{itemize}
\end{thm}

Chowdhury and Patk\'{o}s \cite{shadow} proved a vector space analog of the Lov\'{a}sz's theorem~\cite{Lo} on shadows of sets. In the process, they established the following theorem.
\begin{thm}\label{k+1}{\rm(\cite[Theorem 2.2]{shadow})}
Let $\mathcal{H}\subseteq \left[V\atop k\right]$ and let $m\geq k$ be the real number which satisfies $|\mathcal{H}|=\left[m\atop k\right]$. Denote $\chi(\mathcal{H})=\left\{T\in \left[V\atop k+1\right]:\left[T\atop k\right]\subseteq \mathcal{H}\right\}$. Then
\begin{equation*}
|\chi(\mathcal{H})|\leq \left[m\atop k+1\right].
\end{equation*}
Moreover, equality holds if and only if $\mathcal{H}=\left[M\atop k\right]$ for some $M\in\left[V\atop m\right], m\in \mathbb{Z}^{+}$.
\end{thm}

In \cite{katona}, Frankl proved the following simple but useful basic fact.

\begin{lem}\label{lem}{\rm(\cite[Lemma 4]{katona})}
For $1\leq a\leq n$, there is a bijection
\begin{equation*}
\psi:\left[V\atop a\right]\rightarrow \left[V\atop n-a\right]
\end{equation*}
such that $A\oplus \psi(A)=V$ holds for all $A\in \left[V\atop a\right]$.
\end{lem}

It is routine to enumerate the subspaces that intersect a given subspace trivially.

\begin{lem}\label{lemc21}{\rm(\cite[Propositions 2.2]{chen})} Let $Z$ be an $m$-dimensional subspace of the $n$-dimensional vector space $V$ over~$\mathbb{F}_{q}$. For a positive integer $l$ with $m+l\leq n$, the number of~$l$-dimensional subspaces~~$W$~of~~$V$~such that ${\rm dim}(Z\cap W)=0$ is $q^{lm}\left[n-m\atop l\right]$.
\end{lem}

For $U\leq V$ and $M\leq U$, define the family
\begin{equation*}
\left[U\atop k,0\right]_{M}=\left\{A\in \left[U\atop k\right]:{\dim}(M\cap A)=0\right\}.
\end{equation*}
By Lemma~\ref{lemc21}, the cardinality of $\left[U\atop k,0\right]_{M}$ is $q^{mk}\left[u-m\atop k\right]$ if ${\rm \dim}(U)=u$ and ${\rm \dim}(M)=m$. We will make use of a partition of $\left[U\atop k,0\right]_{M}$ into $\left[u-m\atop k\right]$ subfamilies in the following two lemmas.

\begin{lem}\label{ca1}
For a subspace $U\in \left[V\atop u\right]$, let $M\in \left[U\atop m\right]$, $W\in \left[U\atop u-m\right]$, $M\cap W=\{\textup{\textbf{0}}\}$ and $W\oplus M=U$. For $1\leq k\leq u-m$, the following hold.
\begin{itemize}
\item[\rm(i)] If $B, B'\in\left[W\atop k\right]$ and $B\neq B'$, then
$\left[B\oplus M\atop k,0\right]_{M}\bigcap\left[B'\oplus M\atop k,0\right]_{M}=\emptyset$.
\item[\rm(ii)]
$\left[U\atop k,0\right]_{M}=\biguplus\limits_{B\in\left[W\atop k\right]}\left[B\oplus M\atop k,0\right]_{M}$, where $\biguplus$ denotes the union of pairwise disjoint sets.
\end{itemize}
\end{lem}
\begin{proof}~(i)~For any $B=\langle b_{1}, b_{2}, \ldots, b_{k}\rangle \in\left[W\atop k\right]$, from the proof of Lemma~\ref{lemc21} in \cite{chen}, we have
\begin{equation}\label{eq200}
\left[B\oplus M\atop k,0\right]_{M}=\left\{\langle b_{1}+v_{1}, b_{2}+v_{2}, \ldots, b_{k}+v_{k}\rangle: v_{i}\in M, 1\leq i\leq k\right\}
\end{equation}
and $$\left|\left[B\oplus M\atop k,0\right]_{M}\right|=q^{mk}.$$

Suppose $B, B'\in \left[W\atop k\right]$ and $B\neq B'$. Let $B=\langle b_{1}, b_{2}, \ldots, b_{k}\rangle, B'=\langle b_{1}', b_{2}', \ldots, b_{k}'\rangle$. If there exists $A\in\left[U\atop k\right]$ such that $$A\in\left[B\oplus M\atop k,0\right]_{M}\bigcap\left[B'\oplus M\atop k,0\right]_{M},$$ then
\begin{equation*}
A=\langle b_{1}+v_{1}, b_{2}+v_{2}, \ldots, b_{k}+v_{k}\rangle=\langle b_{1}'+w_{1}, b_{2}'+w_{2}, \ldots, b_{k}'+w_{k}\rangle
\end{equation*}
by~\eqref{eq200}, where $v_{i}, w_{i}\in M~(1\leq i\leq k)$. Thus, $B\oplus M=A\oplus M=B'\oplus M$ and
\begin{equation*}
B=(B\oplus M)\cap W=(B'\oplus M)\cap W=B'
\end{equation*}
by Lemma~\ref{vv}(i), which is a contradiction to the premise. Hence, the result of (i) follows.

(ii)~From (i), it follows that
\begin{equation*}
\begin{array}{rl}
\left|\bigcup\limits_{B\in\left[W\atop k\right]}\left[B\oplus M\atop k,0\right]_{M}\right|=\left|\biguplus\limits_{B\in\left[W\atop k\right]}\left[B\oplus M\atop k,0\right]_{M}\right|=q^{mk}\left[u-m\atop k\right]=\left|\left[U\atop k,0\right]_{M}\right|.
\end{array}
\end{equation*}
It is obvious that for any $B\in\left[W\atop k\right]$, we have $\left[B\oplus M\atop k,0\right]_{M}\subseteq \left[U\atop k,0\right]_{M}$ and this completes the proof.
\end{proof}

\begin{rmk}\label{r100}
With the notation of Lemma~\ref{ca1}, if $B\in \left[W\atop k\right]$ and $A\in\left[B\oplus M\atop k,0\right]_{M}$, then $\left[A\oplus M\atop k,0\right]_{M}=\left[B\oplus M\atop k,0\right]_{M}$.
\end{rmk}

Now we prove the following result on shadows of vector spaces, which will be very useful in the proof of our main theorem.

\begin{lem}\label{l900}
Let $E\in \left[V\atop 1\right]$, $W\in\left[V\atop n-1\right]$ satisfy that $E\oplus W=V$. Further let $\mathcal{H}\subseteq \left[V\atop k,0\right]_{E}$ and $|\mathcal{H}|=q^{k}\left[m\atop k\right]$, where $m\in \mathbb{R}$. Suppose that for any $A\in \mathcal{H}$, $\left[A\oplus E\atop k,0\right]_{E}\subseteq \mathcal{H}$. Then
\begin{equation*}
\left|\bigtriangleup(\mathcal{H})\right|\geq q^{k-1}\left[m\atop k-1\right].
\end{equation*}
Moreover, equality holds if and only if $\mathcal{H}=\biguplus\limits_{B\in \left[M\atop k\right]}\left[B\oplus E\atop k,0\right]_{E}$ for some $M\in\left[W\atop m\right]$, where $m\in \mathbb{Z}^{+}$.
\end{lem}

\begin{proof}~By Lemma~\ref{ca1}(ii), for $1\leq s\leq n-1$ we can define maps
\begin{equation*}
\varphi_{s}:\left[V\atop s,0\right]_{E}\rightarrow \left[W\atop s\right]~\text{by}~A\mapsto B ~\text{if}~A\in \left[B\oplus E\atop s,0\right]_{E}.
\end{equation*}
It is easy to see
\begin{equation}\label{e003}
\varphi_{s}^{-1}(B)=\left[B\oplus E\atop s,0\right]_{E},~|\varphi_{s}^{-1}(B)|=q^{s}
\end{equation}
for any $B\in \left[W\atop s\right]$ by Lemma~\ref{lemc21}.

Suppose
\begin{equation}\label{401}
\left|\bigtriangleup(\mathcal{H})\right|=q^{k-1}\left[m_{1}\atop k-1\right],
\end{equation}
where $m_{1}\in \mathbb{R}$. Define
\begin{equation*}
\mathcal{G}=\left\{G\in \left[W\atop k-1\right]:\left[G\oplus E\atop k-1,0\right]_{E}\subseteq \bigtriangleup(\mathcal{H})\right\}
\end{equation*}
and let
\begin{equation}\label{607}
\left|\mathcal{G}\right|=\left[m_{2}\atop k-1\right],
\end{equation}
where $m_{2}\in \mathbb{R}$. By~Lemma~\ref{ca1}(i) and equations~\eqref{e003}-\eqref{607},
\begin{equation*}
q^{k-1}\left[m_{2}\atop k-1\right]\overset{\eqref{607}}{=}\left|\biguplus\limits_{G\in \mathcal{G}}\varphi_{k-1}^{-1}(G)\right|\leq \left|\bigtriangleup(\mathcal{H})\right|\overset{\eqref{401}}{=}q^{k-1}\left[m_{1}\atop k-1\right],
\end{equation*}
yielding
\begin{equation}\label{403}
m_{2}\leq m_{1}.
\end{equation}

By Theorem~\ref{k+1} and \eqref{607},
\begin{equation}\label{eq30}
\left|\chi(\mathcal{G})\right|=\left|\left\{T\in \left[W\atop k\right]:\left[T\atop k-1\right]\subseteq \mathcal{G}\right\}\right| \leq \left[m_{2}\atop k\right].
\end{equation}
We claim that $\bigcup\limits_{A\in \mathcal{H}}\{\varphi_{k}(A)\}\subseteq \chi(\mathcal{G})$. For any $A\in\mathcal{H}$ and $G\in \left[\varphi_{k}(A)\atop k-1\right]$, we have
\begin{equation*}
\left[G\oplus E\atop k-1,0\right]_{E}\subseteq \left[\varphi_{k}(A)\oplus E\atop k-1,0\right]_{E}.
\end{equation*}
Since $A\in \left[\varphi_{k}(A)\oplus E\atop k,0\right]_{E}$, we have $\left[\varphi_{k}(A)\oplus E\atop k,0\right]_{E}=\left[A\oplus E\atop k,0\right]_{E}\subseteq \mathcal{H}$ by Remark~\ref{r100}. It is easy to obtain that
\begin{equation*}
\left[\varphi_{k}(A)\oplus E\atop k-1,0\right]_{E}\subseteq \bigtriangleup(\mathcal{H}).
\end{equation*}
Thus $\left[G\oplus E\atop k-1,0\right]_{E}\subseteq \bigtriangleup(\mathcal{H})$ and then
\begin{equation}\label{404}
\bigcup\limits_{A\in \mathcal{H}}\{\varphi_{k}(A)\}\subseteq \chi(\mathcal{G}).
\end{equation}

By~Lemma~\ref{ca1}(i),~\eqref{e003} and using $|\mathcal{H}|=q^{k}\left[m\atop k\right]$ and the property of $\mathcal{H}$,  we have $\left|\bigcup\limits_{A\in \mathcal{H}}\{\varphi_{k}(A)\}\right|=\left[m\atop k\right]$. Then $m\leq m_{2}\leq m_{1}$ by~\eqref{403},~\eqref{eq30},~\eqref{404}. Hence,
\begin{equation}\label{eq80}
\left|\bigtriangleup(\mathcal{H})\right|\overset{\eqref{401}}{=}q^{k-1}\left[m_{1}\atop k-1\right]\geq q^{k-1}\left[m\atop k-1\right].
\end{equation}
This proves the inequality of this lemma.

If the equality in~\eqref{eq80} holds, then $\bigcup\limits_{A\in \mathcal{H}}\{\varphi_{k}(A)\}=\chi(\mathcal{G})$ and the equality in~\eqref{eq30} holds as well. Then
\begin{equation*}
\mathcal{G}=\left[M\atop k-1\right]
\end{equation*}
for some $M\in\left[W\atop m\right]$, $m\in \mathbb{Z}^{+}$ by Theorem~\ref{k+1} and
\begin{equation*}
\bigcup\limits_{A\in \mathcal{H}}\{\varphi_{k}(A)\}=\chi(\mathcal{G})=\left[M\atop k\right].
\end{equation*}
Hence, $\mathcal{H}=\biguplus\limits_{B\in \left[M\atop k\right]}\varphi_{k}^{-1}(B)=\biguplus\limits_{B\in \left[M\atop k\right]}\left[B\oplus E\atop k,0\right]_{E}$.

Conversely, let $\mathcal{H}=\biguplus\limits_{B\in \left[M\atop k\right]}\left[B\oplus E\atop k,0\right]_{E}$ for some $M\in\left[W\atop m\right]$. Then $\mathcal{H}=\left[M\oplus E\atop k,0\right]_{E}$ by~Lemma~\ref{ca1}(ii). It follows that $\bigtriangleup(\mathcal{H})=\left[M\oplus E\atop k-1,0\right]_{E}$ and then $\left|\bigtriangleup(\mathcal{H})\right|=q^{k-1}\left[m\atop k-1\right]$ by Lemma~\ref{lemc21}. This completes the proof.
\end{proof}

\section{ Proof of Theorem~\ref{vsc}}

Erd\H{o}s, Frankl and F\"{u}redi \cite{two} gave a short proof of Theorem~\ref{sc} for cover-free families of finite sets. In this section, we adapt their argument to prove Theorem~\ref{vsc}. As we shall see, the problem becomes much more difficult when we deal with vector spaces rather than finite sets. 

As has been shown in Section 1, to prove Theorem~\ref{sc}, we only need to treat part (i). To be more precise, we need to establish the upper bound of a cover-free family with $k=2t$ and then characterize the structure of the maximum cover-free family.

Throughout this section we always suppose that $\mathcal{F}\subseteq \left[V\atop 2t\right]$ is a cover-free  family with ${\rm dim}(V)=n\geq 2t+1$. The following notation and definitions will be adopted in this section.
\begin{itemize}
   \item[$\bullet$]$\mathcal{F}$ denotes a cover-free family in $\left[V\atop 2t\right]$, where $n\geq 2t+1$.
  \item[$\bullet$] Denote $\mathcal{F}_{T}=\{F\in \mathcal{F}:T\leq F\}$ for any subspace $T\leq V$.
   \item[$\bullet$] If $|\mathcal{F}_{T}|=1$, $T\leq F\in \mathcal{F},$ we say that $T$ is a {\it private} subspace of $F$.
   \item[$\bullet$] For any $F\in \mathcal{F}$, let $M_{t}(F)=\left\{T\in \left[F\atop t\right]: |\mathcal{F}_{T}|=1\right\}$.
   \item[$\bullet$] For any $F\in \mathcal{F}$, let $N_{t}(F)=\left\{T\in\left[F\atop t\right]:|\mathcal{F}_{T}|>1\right\}$.
\end{itemize}

We give a simple but very useful criterion to judge private subspaces in a cover-free family.

\begin{lem}\label{r1}
For any $A, B\leq V$ and $A+B=F\in \mathcal{F}$, we have $|\mathcal{F}_{A}|=1$ or $|\mathcal{F}_{B}|=1$.
\end{lem}

\begin{proof}~Note that if $F=A+B$, then $F=F\cap A+F\cap B$. If $|\mathcal{F}_{A}|>1$ and $|\mathcal{F}_{B}|>1$, then there exist $F_{1}\in \mathcal{F}_{A}, F_{2}\in \mathcal{F}_{B}$ with $F_{1}, F_{2}\neq F$ such that $F=F\cap A+F\cap B\leq F\cap F_{1}+ F\cap F_{2}$, a contradiction.
\end{proof}

 \subsection{ The proof of inequality}

\begin{lem}\label{s}
For any $F\in \mathcal{F}$, we have
\begin{equation*}
\left|M_{t}(F)\right|\geq q^{t}\left[2t-1\atop t\right], \left|N_{t}(F)\right|\leq \left[2t-1\atop t-1\right].
\end{equation*}
Moreover, equalities hold if and only if $N_{t}(F)$ is an intersecting family of size $\left[2t-1\atop t-1\right]$.
\end{lem}

\begin{proof}~First we let $t=1$. If $N_{1}(F)=\emptyset$, then $\left|M_{1}(F)\right|=\left[2\atop 1\right]>q$. If $N_{1}(F)\neq \emptyset$, then there exists $A\in N_{1}(F)$. Since for any $B\in \left[F\atop 1\right]\backslash \{A\}$, we have $A\oplus B=F$ and then $|\mathcal{F}_{B}|=1$ by Lemma~\ref{r1}. It then follows that $\left|M_{1}(F)\right|=\left[2\atop 1\right]-1=q$. This proves the case of $t=1$.

Next we treat the case $t\geq 2$. For any $T, T'\in N_{t}(F)\subseteq \left[F\atop t\right]$, since $|\mathcal{F}_{T}|>1$ and $|\mathcal{F}_{T'}|>1$, $T\oplus T'=F$ is impossible by Lemma~\ref{r1}. Then we have that $N_{t}(F)$ is an intersecting family in $F$ and thus $|N_{t}(F)|\leq \left[2t-1\atop t-1\right]$ by Theorem~\ref{EKR}. Hence
\begin{equation*}
\left|M_{t}(F)\right|=\left[2t\atop t\right]-|N_{t}(F)|\geq q^{t}\left[2t-1\atop t\right].
\end{equation*}
If the equality holds, then $|N_{t}(F)|=\left[2t-1\atop t-1\right]$.

It is obvious that if $N_{t}(F)$ is an intersecting family and $\left|N_{t}(F)\right|=\left[2t-1\atop t-1\right]$, then $\left|M_{t}(F)\right|=q^{t}\left[2t-1\atop t\right]$.
\end{proof}

For any $F\in \mathcal{F}$, $T\in\left[F\atop t\right]$ and $t\geq 2$, fix a subspace $R\in \left[V\atop n-t\right]$ with $R\oplus T=V$ and define
\begin{equation*}
\mathcal{A}(F,T)=\left\{
\begin{array}{rl}A\in \left[V\atop t\right]:\!\!\!&{\rm dim}(A\cap T)=0, {\rm dim}(A\cap F)=t-1, \exists E\in \left[F\cap R\atop 1\right] \\[.3cm]
& \text {such that~} |\mathcal{F}_{T\oplus E}|>1 \text{~and~} F=T\oplus E\oplus A\cap F
\end{array}\right\}.
\end{equation*}
It is obvious that ${\dim}(F\cap R)=t$.

\begin{lem}\label{le20}
Let $t\geq 2$ and
\begin{equation}\label{620}
t_{0}=\min\left\{\left[t\atop 1\right], q^{3t-n-1}\left[n-2t+1\atop 1\right]\right\}.
\end{equation}
For any $A, T\in \left[V\atop t\right]$, we have
\begin{equation*}
\left|\left\{F\in \mathcal{F}_{T}:A\in \mathcal{A}(F,T)\right\}\right|\leq t_{0}.
\end{equation*}
\end{lem}

\begin{proof}~
For any $A, T\in \left[V\atop t\right]$, if $A\in \mathcal{A}(F,T)$, then $F=T\oplus E\oplus A\cap F\in \mathcal{F}$ for some $E\in $$\left[F\atop 1\right]$ and $|\mathcal{F}_{T\oplus E}|>1$. We have $|\mathcal{F}_{A\cap F}|=1$ by Lemma~\ref{r1}. Thus the number of possibilities for $F$ is at most $\left[t\atop t-1\right]$ as ${\rm dim}(A\cap F)=t-1$. Note that $\left[t\atop 1\right]\leq q^{3t-n-1}\left[n-2t+1\atop 1\right]$ if and only if $n\geq 3t-1$. Thus the result holds if $n\geq 3t-1$.

In the following we let $n<3t-1$ and prove that
\begin{equation*}
\left|\left\{F\in \mathcal{F}_{T}:A\in \mathcal{A}(F,T)\right\}\right|\leq q^{3t-n-1}\left[n-2t+1\atop 1\right].
\end{equation*}
For any fixed $F_{1}\in \mathcal{F}$, $T\in\left[F_{1}\atop t\right]$ and $A\in \mathcal{A}(F_{1},T)$, there exists $E\in \left[F_{1}\cap R\atop 1\right]$ and $F_{2}\in \mathcal{F}_{T\oplus E}\big\backslash\{F_{1}\}$. 
Then
\begin{equation}\label{er}
{\dim}(F_{1}\cap F_{2})\geq 4t-n\geq t+2>{\dim}(T\oplus E).
\end{equation}
Since $T\oplus E\leq F_{1}\cap F_{2}$ and $F_{1}=A\cap F_{1}\oplus(T\oplus E)$, we have $F_{1}\cap F_{2}=A\cap F_{1}\cap F_{2}\oplus(T\oplus E)$ by Lemma~\ref{vv}(i), then
\begin{equation}\label{t}
{\dim}(A\cap F_{1}\cap F_{2})\overset{\eqref{er}}{\geq }4t-n-(t+1)=3t-n-1\geq 1
\end{equation}
as ${\dim}(T\oplus E)=t+1$.

By Lemma~\ref{lem}, there is a bijection
\begin{equation*}
\psi:\left[A\atop 1\right]\rightarrow \left[A\atop t-1\right]
\end{equation*}
such that $B\oplus \psi(B)=A$ holds for all $B\in \left[A\atop 1\right]$. For any $G\in \left[A\cap F_{1}\cap F_{2}\atop 1\right]$, it is obvious that $\psi(G)\nleq F_{i}$, $i=1,2$; otherwise, $\psi(G)\oplus G\oplus T\oplus E\leq F_{i}$, a contradiction. Meanwhile, $\psi(G)$ cannot be contained in a member $F_{3}\in \mathcal{F}\backslash \{F_{1}, F_{2}\}$; otherwise,
\begin{align}
F_{1}&=A\cap F_{1}\oplus(T\oplus E)\notag\\
&\leq (G\oplus \psi(G))\cap F_{1}+F_{1}\cap F_{2}\notag\\
&= G\oplus \psi(G)\cap F_{1}+F_{1}\cap F_{2}\tag{{\rm by~Lemma \ref{vv}(i)}}\\
&=\psi(G)\cap F_{1}+F_{1}\cap F_{2}\notag\\
&\leq F_{1}\cap F_{3}+F_{1}\cap F_{2},\notag
\end{align}
which is a contradiction to the definition of a cover-free family.

Because $|\mathcal{F}_{A\cap F}|=1$, $\psi$ is a bijection and by~\eqref{t}, the number of subspaces $F\in \mathcal{F}$ such that $A\in \mathcal{A}(F,T)$ is at most
\begin{equation*}
\begin{array}{rl}
\left|\left[A\atop t-1\right]\Big\backslash\left\{\psi(G):G\in \left[A\cap F_{1}\cap F_{2}\atop 1\right]\right\}\right|
\!\!\!\!&=\left[t\atop 1\right]-\left|\left[A\cap F_{1}\cap F_{2}\atop 1\right]\right|\\[.3cm]
&\leq \left[t\atop 1\right]-\left[3t-n-1\atop 1\right]\\[.3cm]
&=q^{3t-n-1}\left[n-2t+1\atop 1\right].
\end{array}
\end{equation*}
\end{proof}

\begin{lem}\label{2}
Let $t\geq 2$. For any $A\in \left[V\atop t\right]$, we have
\begin{equation*}
\left|\left\{(F, T):A\in \mathcal{A}(F,T), F\in \mathcal{F}, T\in\left[F\atop t\right]\right\}\right|\leq t_{0}q^{t(t-1)}\left[t+1\atop 1\right],
\end{equation*}
where $t_{0}$ is defined in~\eqref{620}.
\end{lem}

\begin{proof}~Since ${\rm dim}((A\cap F)\cap T)=0$ and $T\in \left[F\atop t\right]$, for fixed $A$ we have at most $q^{t(t-1)}\left[t+1\atop t\right]$ possibilities for $T$ by Lemma~\ref{lemc21}. Applying Lemma~\ref{le20} yields the conclusion.
\end{proof}

\begin{lem}\label{3}
Suppose $\mathcal{F}\subseteq \left[V\atop 2t\right]$ is a cover-free  family, $T\in\left[V\atop t\right]$ and $t\geq 2$. Then
\begin{equation*}
\sum\limits_{F\in \mathcal{F}_{T}}|\mathcal{A}(F,T)|\geq q^{t^{2}+1}\left[n-2t\atop 1\right]\left(\left[t\atop 1\right]\left|\mathcal{F}_{T}\right|-\left[n-t\atop 1\right]\right).
\end{equation*}
\end{lem}

\begin{proof}~
Suppose $T\in\left[V\atop t\right]$, $F\in \mathcal{F}_{T}$, $R\in \left[V\atop n-t\right]$ and $R\oplus T=V$. For any $E\in\left[F\cap R\atop 1\right]$, there is a bijection
\begin{equation*}
\psi:\left[F\cap R\atop 1\right]\rightarrow \left[F\cap R\atop t-1\right]
\end{equation*}
such that $E\oplus \psi(E)=F\cap R$ by Lemma~\ref{lem}.

For two distinct $E, E'\in\left[F\cap R\atop 1\right]$ with $|\mathcal{F}_{T\oplus E}|>1$ and $|\mathcal{F}_{T\oplus E'}|>1$, we have $\psi(E)$, $\psi(E')\in \left[F\cap R\atop t-1\right]\subseteq \left[R\atop t-1\right]$ and then
\begin{equation}\label{130}
\left[\psi(E)\oplus T\atop t-1,0\right]_{T}\bigcap\left[\psi(E')\oplus T\atop t-1,0\right]_{T}=\emptyset
\end{equation}
by Lemma~\ref{ca1}(i).

We claim that
\begin{equation}\label{131}
{\dim}(B\cap (E\oplus T))=0
\end{equation}
for all $B\in \left[\psi(E)\oplus T\atop t-1,0\right]_{T}$. If ${\dim}(B\cap (E\oplus T))\neq 0$, then there exists $b\in B\cap (E\oplus T)$ and $b\neq \textup{\textbf{0}}$ such that
\begin{equation*}
 b=e+a,
\end{equation*}
where $e\in E$ and $a\in T$. Since $b\in B\in \left[\psi(E)\oplus T\atop t-1,0\right]_{T}$, we have
\begin{equation*}
b=c+d,
\end{equation*}
where $c\in \psi(E), d\in T$ and $c\neq \textup{\textbf{0}}$. Then $e+a=c+d$, consequently
\begin{equation*}
e-c=d-a\in (F\cap R)\cap T=\{\textup{\textbf{0}}\}.
\end{equation*}
Hence, $e=c\in E\cap \psi(E)=\{\textup{\textbf{0}}\}$, a contradiction.

Now we may produce some members of $\mathcal{A}(F,T)$ in three steps. Firstly, we choose $E\in\left[F\cap R\atop 1\right]$ such that $|\mathcal{F}_{T\oplus E}|>1$. Secondly, we choose $B\in \left[\psi(E)\oplus T\atop t-1,0\right]_{T}$. Finally, we form a subset
$$\mathcal{A}_{B}=\left\{A\in \left[V\atop t\right]\Big\backslash \left[F\atop t\right]:B\leq A\right\}\subseteq \mathcal{A}(F,T).$$
Applying~\eqref{130}, \eqref{131} and noting that $\mathcal{A}_{B}\cap \mathcal{A}_{B'}=\emptyset$ if the chosen $B$ and $B'$ are distinct, we do not produce any repeated element of $\mathcal{A}(F,T)$ by these steps and hence we have
\begin{equation*}
\begin{array}{rl}
|\mathcal{A}(F,T)|
\!\!\!\!&\geq\left|\left\{E\in\left[F\cap R\atop 1\right]: |\mathcal{F}_{T\oplus E}|>1\right\}\right|\cdot \left|\left[\psi(E)\oplus T\atop t-1,0\right]_{T}\right|\cdot\left|\left\{A\in\left[V\atop t\right]\big\backslash\left[F\atop t\right]: B\leq A\right\}\right|\\[.3cm]
&=\left|\left\{E\in\left[F\cap R\atop 1\right]:|\mathcal{F}_{T\oplus E}|>1\right\}\right|\cdot q^{t(t-1)}\cdot \left(\left[n-(t-1)\atop 1\right]-\left[2t-(t-1)\atop 1\right]\right)\\[.3cm]
&=q^{t^{2}+1}\left[n-2t\atop 1\right]\left|\left\{E\in\left[F\cap R\atop 1\right]: |\mathcal{F}_{T\oplus E}|>1\right\}\right|,
\end{array}
\end{equation*}
where $B\in \left[\psi(E)\oplus T\atop t-1,0\right]_{T}$ and the size of $\left[\psi(E)\oplus T\atop t-1,0\right]_{T}$ follows from Lemma~\ref{lemc21}.

Then
\begin{equation}\label{8}
\sum\limits_{F\in \mathcal{F}_{T}}|\mathcal{A}(F,T)|\geq q^{t^{2}+1}\left[n-2t\atop 1\right]\sum\limits_{F\in \mathcal{F}_{T}}\left|\left\{E\in\left[F\cap R\atop 1\right]:|\mathcal{F}_{T\oplus E}|>1\right\}\right|.
\end{equation}

For any $T\in \left[V\atop t\right]$, we calculate
\begin{equation*}
\left|\left\{(E, F):E\in\left[F\cap R\atop 1\right], F\in \mathcal{F}_{T}, |\mathcal{F}_{T\oplus E}|>1\right\}\right|
\end{equation*}
in two ways and get the following equation
\begin{equation}\label{7}
\begin{array}{rl}
\sum\limits_{F\in \mathcal{F}_{T}}\left|\left\{E\in\left[F\cap R\atop 1\right]: |\mathcal{F}_{T\oplus E}|>1\right\}\right|
\!\!\!&=\sum\limits_{E\in \left[R\atop 1\right]\atop \left|\mathcal{F}_{T\oplus E}\right|>1}\left|\left\{F\in \mathcal{F}_{T}:T\oplus E\leq F\right\}\right|\\[.3cm]
&=\sum\limits_{E\in \left[R\atop 1\right]\atop \left|\mathcal{F}_{T\oplus E}\right|>1}\left|\mathcal{F}_{T\oplus E}\right|\\[.3cm]
&\geq \sum\limits_{E\in\left[R\atop 1\right]}\left(\left|\mathcal{F}_{T\oplus E}\right|-1\right)\\[.3cm]
&=\sum\limits_{E\in\left[R\atop 1\right]}\left|\mathcal{F}_{T\oplus E}\right|-\left[n-t\atop 1\right].
\end{array}
\end{equation}

For any $T\in \left[V\atop t\right]$, we calculate
\begin{equation*}
\left|\left\{(E, F):E\in\left[F\cap R\atop 1\right], T\leq F, E\leq F \right\}\right|
\end{equation*}
in two ways and get the following equation
\begin{equation}\label{9}
\sum\limits_{E\in\left[R\atop 1\right]}\left|\mathcal{F}_{T\oplus E}\right|=\sum\limits_{F\in\mathcal{F}_{T}}\left|\left\{E:E\in\left[F\cap R\atop 1\right]\right\}\right|=\left[t\atop 1\right]\left|\mathcal{F}_{T}\right|.
\end{equation}
Hence, the result holds by~\eqref{8},~\eqref{7} and~\eqref{9}.
\end{proof}

For every pair $(F,T)$, where $F\in \mathcal{F}$, $T\in \left[F\atop t\right]$ and $t\geq 1$, we define a nonnegative weight function on $\left[V\atop t\right]$, that is, $\omega_{(F,T)}:\left[V\atop t\right] \rightarrow \mathbb{R}$ as follows.

If $ |\mathcal{F}_{T}|=1$, then let
\begin{equation*}
\omega_{(F,T)}(A)\!=\!
\begin{cases}
1, &\text{if $A=T$},\\[.3cm]
0, &\text{otherwise}.
\end{cases}\
\end{equation*}

If $1<|\mathcal{F}_{T}|\leq\frac{q^{n-t}-1}{q^{t}-1}$, then let
\begin{equation*}
\omega_{(F,T)}(A)\!=\!
\begin{cases}
\frac{q^{t}-1}{q^{n-t}-1}, &\text{if $A=T$},\\[.3cm]
0, &\text{otherwise}.
\end{cases}\
\end{equation*}

If $|\mathcal{F}_{T}|>\frac{q^{n-t}-1}{q^{t}-1}$, then let
\begin{equation*}
\omega_{(F,T)}(A)\!=\!
\begin{cases}
\frac{1}{|\mathcal{F}_{T}|}, &\text{if $A=T$},\\[.3cm]
\frac{1}{q^{t(t-1)}\left[t+1\atop 1\right]t_{0}}, &\text{if $A\in \mathcal{A}(F',T)$} \text{~for some~} F'\in \mathcal{F}_{T},\\[.3cm]
0,&\text{otherwise},
\end{cases}\
\end{equation*}
where $t_{0}$ is defined in~\eqref{620}.
\begin{lem}\label{4}
Suppose $T\in \left[V\atop t\right]$ and $F\in \mathcal{F}_{T}$. Then
\begin{equation*}
\sum\limits_{A\in \left[V\atop t\right]}\omega_{(F,T)}(A)\geq \frac{q^{t}-1}{q^{n-t}-1}.
\end{equation*}
Moreover, equality holds if and only if $1<|\mathcal{F}_{T}|\leq\frac{q^{n-t}-1}{q^{t}-1}$.
\end{lem}

\begin{proof}~
Suppose $1\leq |\mathcal{F}_{T}|\leq\frac{q^{n-t}-1}{q^{t}-1}$. It is obvious that the inequality holds.

Suppose $|\mathcal{F}_{T}|>\frac{q^{n-t}-1}{q^{t}-1}$. We need to prove the strict inequality. When $t=1$, clearly $\mathcal{F}\subseteq \left[V\atop 2\right]$ and $|\mathcal{F}_{T}|>\frac{q^{n-1}-1}{q-1}=\left[n-1\atop 1\right]$ is impossible. Next we let $t\geq 2$. For any $T\in \left[V\atop t\right]$ and $|\mathcal{F}_{T}|>\frac{q^{n-t}-1}{q^{t}-1}$, we calculate
\begin{equation*}
\left|\left\{(A, F):A\in\mathcal{A}(F,T), F\in\mathcal{F}_{T}\right\}\right|
\end{equation*}
in two ways and get the following equation
\begin{equation}\label{11}
\left|\left\{A:A\in \bigcup\limits_{F\in \mathcal{F}_{T}}\mathcal{A}(F,T)\right\}\right|\cdot\left|\left\{F\in \mathcal{F}_{T}:A\in \mathcal{A}(F,T)\right\}\right|=\sum\limits_{F\in \mathcal{F}_{T}}|\mathcal{A}(F,T)|.
\end{equation}
For brevity we set $|\mathcal{F}_{T}|=d$ and $d_{0}=\min\left\{\left[t\atop 1\right], d, q^{3t-n-1}\left[n-2t+1\atop 1\right]\right\}$. It is obvious that
\begin{equation}\label{ew12}
\left|\left\{F\in \mathcal{F}_{T}:A\in \mathcal{A}(F,T)\right\}\right|\leq d_{0}
\end{equation}
by Lemma~\ref{le20}.
Then
\begin{align}
\sum\limits_{A\in \left[V\atop t\right]}\omega_{(F,T)}(A)
=&\frac{1}{d}+\frac{1}{q^{t(t-1)}\left[t+1\atop 1\right]t_{0}}\left|\left\{A:A\in \bigcup\limits_{F\in \mathcal{F}_{T}}\mathcal{A}(F,T)\right\}\right|\notag\\
\geq&\frac{1}{d}+\frac{1}{q^{t(t-1)}\left[t+1\atop 1\right]t_{0}d_{0}}\sum\limits_{F\in \mathcal{F}_{T}}|\mathcal{A}(F,T)|\tag{\text{by}~\eqref{11},~\eqref{ew12}}\\
\geq&\frac{1}{d}+\frac{q^{t^{2}+1}\left[n-2t\atop 1\right]}{q^{t(t-1)}\left[t+1\atop 1\right]t_{0}d_{0}}\left(\left[t\atop 1\right]d-\left[n-t\atop 1\right]\right)\tag{\text{by~Lemma}~\ref{3}}\\
=&\frac{1}{d}+\frac{q^{t+1}\left[n-2t\atop 1\right]}{\left[t+1\atop 1\right]t_{0}d_{0}}\left(\left[t\atop 1\right]d-\left[n-t\atop 1\right]\right).\notag
\end{align}

If $n\geq 3t-1$, then
\begin{equation*}
\frac{q^{t+1}\left[n-2t\atop 1\right]}{t_{0}}=\frac{q^{t+1}\left[n-2t\atop 1\right]}{\left[t\atop 1\right]}>1.
\end{equation*}
If $n<3t-1$, then
\begin{equation*}
\frac{q^{t+1}\left[n-2t\atop 1\right]}{t_{0}}=\frac{q^{t+1}\left[n-2t\atop 1\right]}{q^{3t-n-1}\left[n-2t+1\atop 1\right]}=\frac{q^{n-2t+2}\left[n-2t\atop 1\right]}{\left[n-2t+1\atop 1\right]}\geq\frac{q^{3}\left[n-2t\atop 1\right]}{\left[n-2t+1\atop 1\right]}>1
\end{equation*}
as $n\geq 2t+1$. Hence,
\begin{equation*}
\begin{array}{rl}
\sum\limits_{A\in \left[V\atop t\right]}\omega_{(F,T)}(A)
\!\!\!\!&>\frac{1}{d}+\frac{\left[t\atop 1\right]d-\left[n-t\atop 1\right]}{\left[t+1\atop 1\right]d_{0}}\\[.3cm]
&=\frac{q^{t}-1}{q^{n-t}-1}-\frac{d(q^{t}-1)-(q^{n-t}-1)}{d(q^{n-t}-1)}+\frac{\left[t\atop 1\right]d-\left[n-t\atop 1\right]}{\left[t+1\atop 1\right]d_{0}}\\[.3cm]
&=\frac{q^{t}-1}{q^{n-t}-1}+\frac{(d(q^{n-t}-1)-(q^{t+1}-1)d_{0})(d(q^{t}-1)-(q^{n-t}-1))}{d(q^{n-t}-1)(q^{t+1}-1)d_{0}}\\[.3cm]
&>\frac{q^{t}-1}{q^{n-t}-1}+\frac{d(q^{n-t}-1)-(q^{t+1}-1)d_{0}}{d(q^{n-t}-1)(q^{t+1}-1)d_{0}}\\[.3cm]
&\geq \frac{q^{t}-1}{q^{n-t}-1}
\end{array}
\end{equation*}
by noting $d=|\mathcal{F}_{T}|>\frac{q^{n-t}-1}{q^{t}-1}$, $d\geq d_{0}$ and $n\geq 2t+1$.
\end{proof}

\noindent\emph{{\textbf{Proof of Theorem~$\rm\ref{vsc}(i)$ for inequality.}}}
For any $A\in \left[V\atop t\right]$, we have
\begin{equation}\label{eq11}
\sum\limits_{T\in \left[V\atop t\right] \atop F\in \mathcal{F}_{T}}\omega_{(F,T)}(A)\leq 1
\end{equation}
by Lemma~\ref{2} and the definition of weight function. Then
\begin{equation}\label{12}
\sum\limits_{A\in \left[V\atop t\right]}\sum\limits_{T\in \left[V\atop t\right] \atop F\in \mathcal{F}_{T}}\omega_{(F,T)}(A)\leq \left[n\atop t\right].
\end{equation}

For any $F\in \mathcal{F}$, by Lemmas~\ref{s} and \ref{4}, we have
\begin{equation}\label{13}
\begin{array}{rl}
\sum\limits_{T\in \left[F\atop t\right]}\sum\limits_{A\in \left[V\atop t\right]}\omega_{(F,T)}(A)
\!\!\!&=\sum\limits_{T\in M_{t}(F)}\sum\limits_{A\in \left[V\atop t\right]}\omega_{(F,T)}(A)+\sum\limits_{T\in N_{t}(F)}\sum\limits_{A\in \left[V\atop t\right]}\omega_{(F,T)}(A)\\[.3cm]
&=\left|M_{t}(F)\right|\cdot 1+\left|N_{t}(F)\right|\cdot\sum\limits_{A\in \left[V\atop t\right]}\omega_{(F,T)}(A)\\[.3cm]
&\geq q^{t}\left[2t-1\atop t\right]\cdot 1+\left[2t-1\atop t-1\right]\cdot\frac{q^{t}-1}{q^{n-t}-1}~~~~~~~~({\rm noting}~\frac{q^{t}-1}{q^{n-t}-1}<1)\\[.3cm]
&=\frac{q^{n}-1}{q^{n-t}-1}\left[2t-1\atop t\right].
\end{array}
\end{equation}
By~\eqref{12} and \eqref{13},
\begin{equation*}
\begin{array}{rl}
|\mathcal{F}|\frac{q^{n}-1}{q^{n-t}-1}\left[2t-1\atop t\right]
\!\!\!&\leq\sum\limits_{F\in \mathcal{F}\atop T\in \left[F\atop t\right]}\sum\limits_{A\in \left[V\atop t\right]}\omega_{(F,T)}(A)\\[.3cm]
&=\sum\limits_{A\in \left[V\atop t\right]}\sum\limits_{T\in \left[V\atop t\right] \atop F\in \mathcal{F}_{T}}\omega_{(F,T)}(A)\\[.3cm]
&\leq \left[n\atop t\right].
\end{array}
\end{equation*}
It then follows that
\begin{equation*}
|\mathcal{F}|\leq \frac{\left[n-1\atop t\right]}{\left[2t-1\atop t\right]}.
\end{equation*}

\qed

\subsection{ The case of equality}
We now characterize the case of equality of Theorem~\ref{vsc}(i). In this subsection, we always let $\mathcal{F}\subseteq \left[V\atop 2t\right]$ be a cover-free family, $2t\leq n-1$ and $|\mathcal{F}|=\frac{\left[n-1\atop t\right]}{\left[2t-1\atop t\right]}$. Let
\begin{equation}\label{45}
\Omega=\left\{A\in \left[V\atop t\right]:|\mathcal{F}_{A}|=1\right\},
\end{equation}
\begin{equation}\label{e45}
\Gamma=\left\{A\in \left[V\atop t\right]:|\mathcal{F}_{A}|=\frac{q^{n-t}-1}{q^{t}-1}\right\}.
\end{equation}
Note that $q^{t}-1\big|q^{n-t}-1$ if an $S_{q}(t, 2t-1, n-1)$ exists.

\begin{lem}\label{l1}The following hold.
\begin{itemize}
\item[\rm(i)] For any $F\in \mathcal{F}$, $\left|M_{t}(F)\right|=q^{t}\left[2t-1\atop t\right]$ and $\left|N_{t}(F)\right|=\left[2t-1\atop t-1\right]$.
\item[\rm(ii)]For any $A\in \left[V\atop t\right]$, $|\mathcal{F}_{A}|$ is either $1$ or $\frac{q^{n-t}-1}{q^{t}-1}$.
\item[\rm(iii)]$\Omega=\biguplus\limits_{F\in \mathcal{F}}M_{t}(F), \Gamma=\bigcup\limits_{F\in \mathcal{F}}N_{t}(F)$.
\item[\rm(iv)]
$|\Omega|=q^{t}\left[n-1\atop t\right], |\Gamma|=\left[n-1\atop t-1\right]$.
\end{itemize}
\end{lem}

\begin{proof}~(i) Since $|\mathcal{F}|=\frac{\left[n-1\atop t\right]}{\left[2t-1\atop t\right]}$, we have equality for every $F\in \mathcal{F}$ in $\eqref{13}$, and thus $F$ has exactly $q^{t}\left[2t-1\atop t\right]$ private $t$-dimensional subspaces and $\left[2t-1\atop t-1\right]$ non-private $t$-dimensional subspaces $T$ for which $1<|\mathcal{F}_{T}|\leq$$\frac{q^{n-t}-1}{q^{t}-1}$ by Lemma~\ref{4}.

(ii) Moreover, we must have equality in~\eqref{eq11} as well, yielding that for any $A\in \left[V\atop t\right]$,
\begin{equation*}
1=\sum\limits_{T\in \left[V\atop t\right] \atop F\in \mathcal{F}_{T}}\omega_{(F,T)}(A)=\sum\limits_{\atop F\in \mathcal{F}_{A}}\omega_{(F,A)}(A)\!=\!
\begin{cases}
|\mathcal{F}_{A}|, &\text{if $|\mathcal{F}_{A}|=1$},\\[.3cm]
\frac{q^{t}-1}{q^{n-t}-1}|\mathcal{F}_{A}|, &\text{if $1<|\mathcal{F}_{A}|\leq\frac{q^{n-t}-1}{q^{t}-1}$}.
\end{cases}\
\end{equation*}
 Hence, $|\mathcal{F}_{A}|$ is either 1 or $\frac{q^{n-t}-1}{q^{t}-1}$ for any $A\in \left[V\atop t\right]$. 

(iii) We have $\Omega\uplus \Gamma=\left[V\atop t\right]$ by (ii). It is obvious that $\Omega=\biguplus\limits_{F\in \mathcal{F}}M_{t}(F)$ by \eqref{45} and $\Gamma=\bigcup\limits_{F\in \mathcal{F}}N_{t}(F)$  by \eqref{e45}.

(iv) By (i) and (iii), we have
 \begin{equation*}
 |\Omega|=q^{t}\left[2t-1\atop t\right]|\mathcal{F}|=q^{t}\left[2t-1\atop t\right]\cdot\frac{\left[n-1\atop t\right]}{\left[2t-1\atop t\right]}=q^{t}\left[n-1\atop t\right].
\end{equation*}
Then
\begin{equation*}
|\Gamma|=\left[V\atop t\right]-|\Omega|=\left[n-1\atop t-1\right].
\end{equation*}
\end{proof}

\begin{cor}\label{c1}
For all $F, F'\in \mathcal{F}$ with $F\neq F'$, we have $\rm{dim}$$(F\cap F')\leq t$ .
\end{cor}

 \begin{proof}~
If there exist $F, F'\in \mathcal{F}$ with $F\neq F'$ such that $\rm{dim}$$(F\cap F')> t$, since $F\cap F'\leq F, F'$, we have that $A$ is a private subspace of $F$ by Lemma~\ref{r1}, where $A\oplus (F\cap F')=F$ and $\rm{dim}$$(A)\leq t-1$. We can choose $B\in \left[V\atop t\right]\big\backslash \left[F\atop t\right]$ such that $A\leq B$ and $|\mathcal{F}_{B}|\leq 1$. By Lemma~\ref{l1}(ii), we have $|\mathcal{F}_{B}|=1$. Then $\mathcal{F}_{B}=\{F\}$ by $A\leq B$ and $A\leq F\in \mathcal{F}$. It is a contradiction to $B\in \left[V\atop t\right]\big\backslash \left[F\atop t\right]$.
 \end{proof}

When $t=1$, we are ready to prove the case of equality in Theorem~\ref{vsc}(i).
\begin{pro}\label{p1}
Suppose $\mathcal{F}\subseteq \left[V\atop 2\right]$ is a cover-free  family. If $|\mathcal{F}|=\left[n-1\atop 1\right]$, then
\begin{equation*}
\mathcal{F}=\{E\oplus S: S\in \mathcal{S}\}
\end{equation*}
for some $E\in \left[V\atop 1\right]$, where $\mathcal{S}$ is a $q$-Steiner system $S_{q}(1, 1, n-1)$ on some $W\in \left[V\atop n-1\right]$ satisfying $E\oplus W=V$.
\end{pro}

\begin{proof}~
For any $F\in \mathcal{F}$, there exists exactly one $E\in \left[F\atop 1\right]$ such that $|\mathcal{F}_{E}|=\frac{q^{n-1}-1}{q-1}=|\mathcal{F}|$ by Lemma~\ref{l1}(i). Hence,
\begin{equation*}
\mathcal{F}_{E}=\left\{F\in \left[V\atop 2\right]: E\leq F\right\}
\end{equation*}
and
\begin{equation*}
\mathcal{F}=\left\{E\oplus S: S\in \left[W\atop 1\right]\right\}
\end{equation*}
where $E\oplus W=V$. It is obvious that $\left[W\atop 1\right]$ is a $q$-Steiner system $S_{q}(1, 1, n-1)$.
\end{proof}

In the rest of this section we consider the case $t\geq 2$. For any $B\in \left[V\atop t-1\right]$, we fix a subspace $M\in \left[V\atop n-t+1\right]$ such that $V=B\oplus M$. Define
\begin{equation*}
\Omega(B)=\left\{E\in \left[M\atop 1\right]:E\oplus B\in\Omega \right\},
\end{equation*}
\begin{equation*}
\Gamma(B)=\left\{E\in \left[M\atop 1\right]:E\oplus B\in\Gamma \right\},
\end{equation*}
\begin{equation*}
b_{i}=\left|\left\{F\in \mathcal{F}:B\leq F, \left|\left[F\atop 1\right]\bigcap \Gamma(B)\right|=i \right\}\right|,
\end{equation*}
where $\Omega$ and $\Gamma$ are defined in \eqref{45} and \eqref{e45}.
It is obvious that
\begin{equation}\label{eeq}
\left|\Omega(B)\right|+\left|\Gamma(B)\right|=\left[n-t+1\atop 1\right].
\end{equation}
If $b_{i}\neq 0$, then $0\leq i\leq \left[t+1\atop 1\right]$, because $\rm{dim}$$(F\cap M)=t+1$ for any $F\in\mathcal{F}_{B}$.

\begin{lem}\label{l2}
For any $B\in \left[V\atop t-1\right]$, we have $|\Gamma(B)|\in\left\{0, 1, \left[n-t+1\atop 1\right]\right\}$.
\end{lem}

\begin{proof}~
For any $B\in \left[V\atop t-1\right]$, we calculate
\begin{equation*}
\left|\left\{(E, F):E\in\left[F\atop 1\right]\bigcap \Gamma(B), B\leq F\in\mathcal{F}\right\}\right|
\end{equation*}
in three ways and get the following equation
\begin{equation}\label{eq12}
\sum\limits_{i=0}^{\left[t+1\atop 1\right]}ib_{i}=\sum\limits_{F\in\mathcal{F}_{B}}\left|\left[F\atop 1\right]\bigcap \Gamma(B)\right|=\sum\limits_{E\in \Gamma(B)}\left|\mathcal{F}_{E\oplus B}\right|=\frac{q^{n-t}-1}{q^{t}-1}\left|\Gamma(B)\right|,
\end{equation}
where the last equality holds by Lemma~\ref{l1}(ii) as $E\oplus B$ is not private.

Similarly, since $\rm{dim}$$(F\cap M)=t+1$ holds for any $F\in\mathcal{F}_{B}$, we have
\begin{equation}\label{eq13}
\begin{array}{rl}
\sum\limits_{i=0}^{\left[t+1\atop 1\right]}\left(\left[t+1\atop 1\right]-i\right)b_{i}
=&\!\!\!\sum\limits_{F\in\mathcal{F}_{B}}\left|\left[F\atop 1\right]\bigcap \left(\left[M\atop 1\right]\big\backslash\Gamma(B)\right)\right|\\[.3cm]
=&\!\!\!\sum\limits_{F\in\mathcal{F}_{B}}\left|\left[F\atop 1\right]\bigcap \Omega(B)\right|\\[.3cm]
=&\!\!\!\sum\limits_{E\in \Omega(B)}\left|\mathcal{F}_{E\oplus B}\right|\\[.3cm]
=&\!\!\!\left|\Omega(B)\right|\\[.3cm]
\overset{\eqref{eeq}}{=}\!\!&\!\!\!\left[n-t+1\atop 1\right]-\left|\Gamma(B)\right|.
\end{array}
\end{equation}

We will prove that $\left|\mathcal{F}_{E\oplus E'\oplus B}\right|=1$ for any two distinct $E, E'\in \Gamma(B)\subseteq \left[M\atop 1\right]$. For any $E\in \Gamma(B)$, we have $E\oplus B\in \Gamma$, so $\left|\mathcal{F}_{E\oplus B}\right|=\frac{q^{n-t}-1}{q^{t}-1}$ by~\eqref{e45}. For any $F, F'\in \mathcal{F}_{E\oplus B}$, if $F\neq F'$, then $F\cap F'=E\oplus B$ by Corollary~\ref{c1}. Thus $F\cap F'\cap M=E$. Since
\begin{equation*}
\left|\bigcup\limits_{F\in \mathcal{F}_{E\oplus B}}\left[F\cap M\atop 1\right]\right|=\frac{q^{n-t}-1}{q^{t}-1}\left(\left[t+1\atop 1\right]-1\right)+1=\left[n-t+1\atop 1\right]=\left|\left[M\atop 1\right]\right|
\end{equation*}
we have
\begin{equation}\label{m}
\left[M\atop 1\right]=\biguplus\limits_{F\in \mathcal{F}_{E\oplus B}}\left(\left[F\cap M\atop 1\right]\big\backslash\{E\}\right)\bigcup \{E\}.
\end{equation}
Then for every $E'\in \Gamma(B)\backslash \{E\}\subseteq \left[M\atop 1\right]\backslash \{E\}$, $\left|\mathcal{F}_{E\oplus E'\oplus B}\right|=1$ by Corollary~\ref{c1}. Hence,
\begin{equation}\label{eq14}
\begin{array}{rl}
\sum\limits_{i=0}^{\left[t+1\atop 1\right]}\binom{i}{2}b_{i}
&=\sum\limits_{F\in\mathcal{F}_{B}}\binom{\left|\left[F\atop 1\right]\bigcap \Gamma(B)\right|}{2}\\[.3cm]
&=\sum\limits_{E, E'\in \Gamma(B)\atop E\neq E'}\left|\mathcal{F}_{E\oplus E'\oplus B}\right|\\[.3cm]
&=\binom{\left|\Gamma(B)\right|}{2}.
\end{array}
\end{equation}
If we subtract the double of~\eqref{eq14} from~\eqref{eq12} multiplied by $q\left[t\atop 1\right]$, we obtain
\begin{equation}\label{eq19}
\sum\limits_{i=0}^{\left[t+1\atop 1\right]}i\left(\left[t+1\atop 1\right]-i\right)b_{i}
=\left|\Gamma(B)\right|\left(\left[n-t+1\atop 1\right]-\left|\Gamma(B)\right|\right).
\end{equation}
Then
\begin{equation*}
\sum\limits_{i=0}^{\left[t+1\atop 1\right]}i\left(\left[t+1\atop 1\right]-i\right)b_{i}
=\left|\Gamma(B)\right|\sum\limits_{i=0}^{\left[t+1\atop 1\right]}\left(\left[t+1\atop 1\right]-i\right)b_{i}
\end{equation*}
by~\eqref{eq13} and~\eqref{eq19}.
Thus,
\begin{equation}\label{eq20}
\sum\limits_{i=0}^{\left[t+1\atop 1\right]}\left(\left|\Gamma(B)\right|-i\right)\left(\left[t+1\atop 1\right]-i\right)b_{i}
=0.
\end{equation}

Case\ 1$:$~Suppose $\left|\Gamma(B)\right|\geq\left[t+1\atop 1\right]$. For $0\leq i\leq\left[t+1\atop 1\right]$, we have
\begin{equation}\label{eq22}
\left(\left|\Gamma(B)\right|-i\right)\left(\left[t+1\atop 1\right]-i\right)b_{i}=0
\end{equation}
by~\eqref{eq20}. So we have $b_{i}=0$ by~\eqref{eq22} for $0\leq i\leq\left[t+1\atop 1\right]-1$. Thus
\begin{equation*}
\left|\Gamma(B)\right|=\left[n-t+1\atop 1\right]
\end{equation*}
by~\eqref{eq13}.

Case\ 2$:$~Suppose $\left|\Gamma(B)\right|<\left[t+1\atop 1\right]$.  Actually, $b_{i}=0$ for $i>\left|\Gamma(B)\right|$. Then for $0\leq i\leq\left|\Gamma(B)\right|$,~\eqref{eq22} holds by~\eqref{eq20}. For $0\leq i\leq\left|\Gamma(B)\right|-1$, $b_{i}=0$ by~\eqref{eq22}. Thus
\begin{equation}\label{eq21}
\left(\left[t+1\atop 1\right]-\left|\Gamma(B)\right|\right)b_{\left|\Gamma(B)\right|}=\left[n-t+1\atop 1\right]-\left|\Gamma(B)\right|
\end{equation}
by~\eqref{eq13}. If $\left|\Gamma(B)\right|\geq 2$, then $b_{\left|\Gamma(B)\right|}=1$ by~\eqref{eq14}, hence $\left[t+1\atop 1\right]=\left[n-t+1\atop 1\right]$ by~\eqref{eq21}, which leads to a contradiction $n=2t$. Thus $\left|\Gamma(B)\right|=0$ or $1$ in this case.
\end{proof}

\begin{lem}\label{ff}
For any $B\in \bigtriangleup(\Omega)$, we have
\begin{equation*}
\left|\mathcal{F}_{B}\right|\leq\frac{q^{n-t}-1}{q^{t}-1}.
\end{equation*}
Moreover, equality holds if and only if $|\Gamma(B)|=1$.
\end{lem}
\begin{proof}~
Suppose that $B\in \bigtriangleup(\Omega)$, $M\in \left[V\atop n-t+1\right]$ and $V=B\oplus M$. By~\eqref{eeq}, Lemma~\ref{l2} and $|\Omega(B)|\geq 1$, we have $|\Gamma(B)|\leq 1$.

Case\ 1$:$~Suppose $|\Gamma(B)|=1$. Then there exists exactly one $E \in \left[M\atop 1\right]$ such that $B\oplus E \in \Gamma$, so
\begin{equation}\label{q12}
\left[M\atop 1\right]=\{E\}\biguplus \Omega(B)
\end{equation}
and
\begin{equation}\label{q13}
\Omega(B)=\biguplus\limits_{F\in \mathcal{F}_{E\oplus B}}\left(\left[F\cap M\atop 1\right]\big\backslash\{E\}\right)
\end{equation}
by~\eqref{m} and~\eqref{q12}. We have $\mathcal{F}_{B}=\mathcal{F}_{E\oplus B}$; otherwise, there exists $F_{1}\in \mathcal{F}_{B}\backslash \mathcal{F}_{E\oplus B}$ and then $\left[F_{1}\cap M\atop 1\right]\subseteq \Omega(B)$. For any fixed $E'\in \left[F_{1}\cap M\atop 1\right]$, there exists $F_{2}\in \mathcal{F}_{E\oplus B}$ such that $E'\leq F_{2}$ by~\eqref{q13}. Consequently, $E'\oplus B\leq F_{1}\cap F_{2}$, which is a
contradiction to $|\mathcal{F}_{E'\oplus B}|=1$ as $E'\in\Omega(B)$. Hence
\begin{equation*}
\left|\mathcal{F}_{B}\right|=\left|\mathcal{F}_{B\oplus E}\right|=\frac{q^{n-t}-1}{q^{t}-1}.
\end{equation*}

Case\ 2$:$~Suppose $|\Gamma(B)|=0$. Then for any $E \in \left[M\atop 1\right]$, $B\oplus E \in \Omega$. Since for any $E \in \left[M\atop 1\right]=\Omega(B)$, there exists $F\in \mathcal{F}_{B}$ such that $E\in \left[F\cap M\atop 1\right]$. Then
\begin{equation}\label{clm}
\left[M\atop 1\right]=\bigcup\limits_{F\in \mathcal{F}_{B}}\left[F\cap M\atop 1\right].
\end{equation}
We claim that for any two distinct $F, F'\in \mathcal{F}_{B}$, $\left[F\cap M\atop 1\right]\bigcap \left[F'\cap M\atop 1\right]=\emptyset$. By Corollary~\ref{c1}, ${\dim}(F\cap F')\leq  t$. If ${\dim}(F\cap F')=t$, then ${\dim}(F\cap F'\cap M)=1$ as $B\leq F\cap F'$ and $B\cap M=\{\textbf{0}\}$. Thus $B\oplus(F\cap F'\cap M)\leq F\cap F'$, which is a contradiction to the premise. Then $F\cap F'=B$. Thus the claim holds by $(F\cap M)\cap (F'\cap M)=B\cap M=\{\textup{\textbf{0}}\}$. Hence,
\begin{equation*}
\left[M\atop 1\right]=\biguplus\limits_{F\in \mathcal{F}_{B}}\left[F\cap M\atop 1\right]
\end{equation*}
by~\eqref{clm} and the above claim. Then
\begin{equation*}
\left|\mathcal{F}_{B}\right|=\frac{\left[n-t+1\atop 1\right]}{\left[t+1\atop 1\right]}<\frac{q^{n-t}-1}{q^{t}-1}.
\end{equation*}
as ${\dim}(M)=n-t+1$ and ${\dim}(F\cap M)=t+1$ for any $F\in \mathcal{F}_{B}$.
\end{proof}

\begin{lem}\label{cc}
We have
\begin{equation}\label{eq33}
\left|\bigtriangleup(\Omega)\right|\geq q^{t-1}\left[n-1\atop t-1\right].
\end{equation}
Moreover, if the equality holds, then $N_{t}(F)=\left\{H\in\left[F\atop t\right]:E\leq H\right\}$ for any $F\in\mathcal{F}$ and some $E\in\left[F\atop 1\right]$.
\end{lem}
\begin{proof}~
Since $|\mathcal{F}|=\frac{\left[n-1\atop t\right]}{\left[2t-1\atop t\right]}$, for any $F\in\mathcal{F}$, we have
\begin{equation*}
\left|M_{t}(F)\right|=\left|\left\{T\in \left[F\atop t\right]: |\mathcal{F}_{T}|=1\right\}\right|= q^{t}\left[2t-1\atop t\right],
\end{equation*}
\begin{equation*}
\left|N_{t}(F)\right|=\left|\left\{T\in\left[F\atop t\right]:|\mathcal{F}_{T}|>1\right\}\right|=\left[2t-1\atop t-1\right]
\end{equation*}
by Lemma~\ref{l1}(i). By Lemma~\ref{s}, $N_{t}(F)\subseteq \left[F\atop t\right]$ is an intersecting family. Then we have
\begin{equation*}
N_{t}(F)=\left\{H\in\left[F\atop t\right]:E\leq H\right\}
\end{equation*}
for some $E\in\left[F\atop 1\right]$ or
\begin{equation*}
N_{t}(F)=\left[Y\atop t\right]
\end{equation*}
for some $Y\in\left[F\atop 2t-1\right]$ by Theorem~\ref{EKR}. Then we calculate the size of $\bigtriangleup(M_{t}(F))$ for any $F\in\mathcal{F}$ in the following two cases.

Case\ 1$:$~Suppose $N_{t}(F)=\left\{H\in\left[F\atop t\right]:E\leq H\right\}$ for some $E\in\left[F\atop 1\right]$. Then
\begin{equation*}
M_{t}(F)=\left\{H\in\left[F\atop t\right]:E\nleq H\right\}
\end{equation*}
and
\begin{equation*}
\bigtriangleup(M_{t}(F))=\left\{G\in\left[F\atop t-1\right]:E\nleq G\right\}.
\end{equation*}
By Lemma~\ref{lemc21},
\begin{equation*}
\left|\bigtriangleup(M_{t}(F))\right|= q^{t-1}\left[2t-1\atop t-1\right].
\end{equation*}

Case\ 2$:$~Suppose $N_{t}(F)\!=\!\left[Y\atop t\right]$ for some $Y\!\in\!\left[F\atop 2t-1\right]$. Then there exists $E'\in \left[F\atop 1\right]$ such that $E'\oplus Y=F$. We claim that
\begin{equation}\label{ea}
M_{t}(F)=\!\left\{H\in\left[F\atop t\right]:H\nleq Y\right\}=\bigcup\limits_{T\in \left[Y\atop t\right]}\left(\left[T\oplus E'\atop t\right]\big\backslash\{T\}\right).
\end{equation}
For any $H\in M_{t}(F)$, since ${\dim}(H\cap Y)=t-1$, we have $H=\langle y_{1}, y_{2}, \ldots y_{t-1}, y_{t}+e\rangle$, where $y_{i}\in Y$, $e\in E'$, $y_{1}, y_{2}, \ldots y_{t-1}, y_{t}$ are linearly independent and $e\neq \textup{\textbf{0}}$. Then $H\in \bigcup\limits_{T\in \left[Y\atop t\right]}\left(\left[T+E'\atop t\right]\big\backslash\{T\}\right)$. It is obvious that $\bigcup\limits_{T\in \left[Y\atop t\right]}\left(\left[T\oplus E'\atop t\right]\big\backslash\{T\}\right)\subseteq M_{t}(F)$.

By~\eqref{ea}, for any $G\in \left[Y\atop t-1\right]$, we have
\begin{equation*}
G\oplus E'\in \bigcup\limits_{T\in \left[Y\atop t\right]}\left(\left[T\oplus E'\atop t\right]\big\backslash\{T\}\right)=M_{t}(F),
\end{equation*}
that is, $G\in \bigtriangleup(M_{t}(F))$. Then $\left[Y\atop t-1\right]\subseteq \bigtriangleup(M_{t}(F))$. Since $\left[F\atop t\right]\big\backslash \left[Y\atop t\right]\subseteq M_{t}(F)$, we also have $\left[F\atop t-1\right]\big\backslash \left[Y\atop t-1\right]\subseteq \bigtriangleup(M_{t}(F))$. Then
\begin{equation*}
\left|\bigtriangleup(M_{t}(F))\right|=\left|\left[F\atop t-1\right]\right|=\left[2t\atop t-1\right]> q^{t-1}\left[2t-1\atop t-1\right]
\end{equation*}
and hence
\begin{equation}\label{43}
\left|\bigtriangleup(M_{t}(F))\right|\geq q^{t-1}\left[2t-1\atop t-1\right].
\end{equation}
Note that the equality occurs only in Case\ 1.

For any $B\in \bigtriangleup(\Omega)$, we have $|\mathcal{F}_{B}|\leq \frac{q^{n-t}-1}{q^{t}-1}$ by Lemma~\ref{ff}. Hence, we calculate
\begin{equation*}
\left|\left\{(B, F):B\in \bigtriangleup(\Omega), B\leq F\in \mathcal{F}\right\}\right|
\end{equation*}
in two ways and get the following inequality
\begin{equation*}
\begin{array}{rl}
\frac{q^{n-t}-1}{q^{t}-1}\left|\bigtriangleup(\Omega)\right|\!\!\!&\geq|\mathcal{F}_{B}|\cdot\left|\bigtriangleup(\Omega)\right|=\left|\bigtriangleup(M_{t}(F))\right|\cdot \left|\mathcal{F}\right|\\[.3cm]
&\geq q^{t-1}\left[2t-1\atop t-1\right]\cdot \frac{\left[n-1\atop t\right]}{\left[2t-1\atop t\right]}\\[.3cm]
&=q^{t-1}\left[n-1\atop t\right]
\end{array}
\end{equation*}
by Lemma~\ref{l1}(iii) and~\eqref{43}. Then we have
\begin{equation*}
\left|\bigtriangleup(\Omega)\right|\geq q^{t-1}\left[n-1\atop t-1\right].
\end{equation*}
If the equality holds, then $$\left|\bigtriangleup(M_{t}(F))\right|=q^{t-1}\left[2t-1\atop t-1\right],\ |\mathcal{F}_{B}|=\frac{q^{n-t}-1}{q^{t}-1}$$ for any $F\in \mathcal{F}$ and $B\in \bigtriangleup(\Omega)$. Hence, we only have Case 1  for all $F\in \mathcal{F}$, completing the proof.
(Note that for any $B\in \bigtriangleup(\Omega)$, $|\mathcal{F}_{B}|=\frac{q^{n-t}-1}{q^{t}-1}$ in Case\ 1. The reason is as follows. For any fixed $B\in \bigtriangleup(\Omega)$, there exists $F\in \mathcal{F}$ such that $B\in \bigtriangleup(M_{t}(F))$ by Lemma~\ref{l1}(iii). Then $B\oplus E\in N_{t}(F)$, where $E$ is defined in Case\ 1. Let $M\in \left[V\atop n-t+1\right]$ such that $V=B\oplus M$. Since ${\rm \dim}((B\oplus E)\cap M)=1$, we have $$B\oplus(B\oplus E)\cap M=B\oplus E\in N_{t}(F),$$ that is, $(B\oplus E)\cap M\in\Gamma(B)$. Then $|\Gamma(B)|=1$ by Lemma \ref{l2} as $B\in \bigtriangleup(\Omega)$. Hence, $|\mathcal{F}_{B}|=\frac{q^{n-t}-1}{q^{t}-1}$ by Lemma~\ref{ff}.)
\end{proof}

\noindent\emph{{\textbf{Proof of Theorem~{\rm\ref{vsc}(i)} for equality.}}}
If $t=1$, by Proposition~\ref{p1}, the result holds. If $t\geq 2$, by Lemma~\ref{cc}, we infer
\begin{equation*}
\left|\left\{B\in \left[V\atop t-1\right]:|\Omega(B)|\geq 1\right\}\right|\geq q^{t-1}\left[n-1\atop t-1\right],
\end{equation*}
equivalently,
\begin{equation}\label{eq34}
x=:\left|\left\{B\in \left[V\atop t-1\right]:|\Gamma(B)|\leq 1\right\}\right|\geq q^{t-1}\left[n-1\atop t-1\right]
\end{equation}
by~\eqref{eeq} and Lemma~\ref{l2}. Using \eqref{eq34} and Lemma~\ref{l1}(iv), we deduce
\begin{equation*}
\begin{array}{rl}
\left[t\atop t-1\right]\left[n-1\atop t-1\right]
=\!\!&\left\{(B, T):B\leq T\in \Gamma, B\in \left[V\atop t-1\right]\right\}\\[.3cm]
=\!\!&\sum\limits_{B\in \left[V\atop t-1\right]}|\Gamma(B)|\\[.3cm]
\leq\!\!&x\cdot 1+\left(\left[n\atop t-1\right]-x\right)\cdot \left[n-t+1\atop 1\right]\\[.3cm]
=\!\!&\left[n\atop t-1\right]\cdot\left[n-t+1\atop 1\right]-\left(\left[n-t+1\atop 1\right]-1\right)x\\[.3cm]
\overset{\eqref{eq34}}{\leq}\!\!\!\!&\left[n-t+1\atop 1\right]\cdot \left(\left[n\atop t-1\right]-q^{t-1}\left[n-1\atop t-1\right]\right)+q^{t-1}\left[n-1\atop t-1\right]\\[.3cm]
=\!\!&\left[t\atop t-1\right]\left[n-1\atop t-1\right].
\end{array}
\end{equation*}
Thus we must have equalities in \eqref{eq34} and consequently in \eqref{eq33}, that is,
\begin{equation}\label{ee}
\left|\bigtriangleup(\Omega)\right|=q^{t-1}\left[n-1\atop t-1\right].
\end{equation}
Hence, for any $F\in \mathcal{F}$,
\begin{equation*}
N_{t}(F)=\left\{H\in\left[F\atop t\right]:E\leq H\right\}
\end{equation*}
for some $E\in\left[F\atop 1\right]$ by Lemma~\ref{cc}.
Then by Lemma~\ref{ca1}(ii),
\begin{equation}\label{e203}
M_{t}(F)=\left\{H\in\left[F\atop t\right]:E\nleq H\right\}=\biguplus\limits_{A\in\left[G\atop t\right]}\left[A\oplus E\atop t,0\right]_{E},
\end{equation}
where $G\in \left[F\atop 2t-1\right]$ and $G\oplus E=F$. 

For $E\in \left[V\atop 1\right]$, define
\begin{equation*}
\mathcal{F}(E)=\left\{F\in \mathcal{F}:E\leq H \text{~for~any~} H\in N_{t}(F)\right\}.
\end{equation*}
Denote
\begin{equation*}
\mathcal{E}=\left\{E\in \left[V\atop 1\right]:\mathcal{F}(E)\neq \emptyset\right\}.
\end{equation*}
For $E\in \mathcal{E}$, denote
\begin{equation}\label{e100}
\Omega_{t}(E)=\biguplus\limits_{F\in \mathcal{F}(E)} M_{t}(F)=\bigcup\limits_{F\in \mathcal{F}(E)}\left\{H\in \left[F\atop t\right]:E\nleq H\right\},
\end{equation}
\begin{equation}\label{e1000}
\Gamma_{t}(E)=\bigcup\limits_{F\in \mathcal{F}(E)} N_{t}(F)=\bigcup\limits_{F\in \mathcal{F}(E)}\left\{H\in \left[F\atop t\right]:E\leq H\right\}.
\end{equation}
It is obvious that $\Omega=\biguplus\limits_{E\in \mathcal{E}}\Omega_{t}(E)$. Next, we will prove  $\bigtriangleup(\Omega)=\biguplus\limits_{E\in \mathcal{E}}\bigtriangleup(\Omega_{t}(E))$.

We claim that $\bigtriangleup(\Omega_{t}(E))\cap\bigtriangleup(\Omega_{t}(E'))=\emptyset$ for any $E, E'\in \Omega_{t}(E)$ with $E\neq E'$. Suppose $P\in \bigtriangleup(\Omega_{t}(E))\cap\bigtriangleup(\Omega_{t}(E'))$. Then $|\Omega(P)|\geq 1$ and $|\Gamma(P)|\leq1$ by~\eqref{eeq} and Lemma~\ref{l2}. By~\eqref{e100}, $E\nleq H$ for any $H\in \Omega_{t}(E)$, $E'\nleq H'$ for any $H'\in \Omega_{t}(E')$, so $E, E' \nleq P$. Then $P\oplus E \in \Gamma_{t}(E)\subseteq \Gamma$, $P\oplus E' \in \Gamma_{t}(E')\subseteq \Gamma$ by~\eqref{e1000}, that is, $E, E'\in \Gamma(P)$, which is a contradiction to $|\Gamma(P)|\leq 1$.

Suppose $|\Omega_{t}(E)|=q^{t}\left\lbrack m_{\!_E}\atop t\right\rbrack$ for $E\in \mathcal{E}$, where $m_{\!_E}\in \mathbb{R}$. Since for any $H\in \Omega_{t}(E)$, we have $H\in M_{t}(F)$ for some $F\in \mathcal{F}(E)$  by~\eqref{e100}. Thus $$\left[H\oplus E\atop t,0\right]_{E}=\left[A\oplus E\atop t,0\right]_{E}\subseteq M_{t}(F) \subseteq \Omega_{t}(E)$$ for some $A\in\left[G\atop t\right]$ by~\eqref{e203}, Remark~\ref{r100} and~\eqref{e100}, where $G\in \left[F\atop 2t-1\right]$ and $G\oplus E=F$. Then
\begin{equation}\label{eq50}
\left|\bigtriangleup(\Omega_{t}(E))\right|\geq q^{t-1}\left[m_{\!_E}\atop t-1\right]
\end{equation}
by Lemma~\ref{l900}.

By Lemma~\ref{l1}(iv), we have $q^{t}\left[n-1\atop t\right]=\left|\Omega\right|=\sum\limits_{E\in \mathcal{E}}\left|\Omega_{t}(E)\right|=q^{t}\sum\limits_{E\in \mathcal{E}}\left[m_{\!_E}\atop t\right]$. Then $m_{\!_E}\leq n-1$ for all $E\in \mathcal{E}$. Hence
\begin{align}
q^{t-1}\left[n-1\atop t-1\right]
\overset{\eqref{ee}}{=}\!\!&\left|\bigtriangleup(\Omega)\right|=\left|\biguplus\limits_{E\in \mathcal{E}}\bigtriangleup(\Omega_{t}(E))\right|\notag\\
=&\sum\limits_{E\in \mathcal{E}}\left|\bigtriangleup(\Omega_{t}(E))\right|\notag\\
\overset{\eqref{eq50}}{\geq}\!\!&\sum\limits_{E\in \mathcal{E}}q^{t-1}\left[m_{\!_E}\atop t-1\right]\label{a1}\\
=&q^{t-1}\sum\limits_{E\in \mathcal{E}}\frac{q^{t}-1}{q^{m_{\!_E}-t+1}-1}\left[m_{\!_E}\atop t\right]\notag\\
\geq&\frac{q^{t-1}(q^{t}-1)}{q^{n-t}-1}\sum\limits_{E\in \mathcal{E}}\left[m_{\!_E}\atop t\right]\label{a2}\\
=&\frac{q^{t-1}(q^{t}-1)}{q^{n-t}-1}\left[n-1\atop t\right]\notag\\
=&q^{t-1}\left[n-1\atop t-1\right]\notag.
\end{align}
This forces that the inequalities \eqref{a1} and \eqref{a2} become equalities. Furthermore, the equality in \eqref{a2} requires that $\mathcal{E}=\{E\}$ and $m_{\!_E}= n-1$. This gives that $\Omega=\Omega_{t}(E)$, $\left|\Omega\right|=\left|\Omega_{t}(E)\right|=q^{t}\left[n-1\atop t\right]$ and the equality holds in~\eqref{a1}. Then $$E\in\bigcap\limits_{F\in \mathcal{F}}N_{t}(F),$$ that is, $E\leq F$ for every $F\in \mathcal{F}$. For $W\in \left[V\atop n-1\right]$ with $E\oplus W=V$, set
\begin{equation*}
\mathcal{S}=\left\{F\cap W: F\in \mathcal{F}\right\}.
\end{equation*}
Then $|\mathcal{S}|=|\mathcal{F}|=\frac{\left[n-1\atop t\right]}{\left[2t-1\atop t\right]}$. For any two distinct $S, S' \in \mathcal{S}$, we have ${\dim}(S\cap S') \leq t-1$ by Corollary~\ref{c1}. Thus $\mathcal{S}$ is a $q$-Steiner system $S_{q}(t, 2t-1, n-1)$ on $W$.
\qed
\medskip

\noindent\emph{{\textbf{Proof of Theorem~{\rm\ref{vsc}}.}}}
Since the subsections 3.1 and 3.2 have treated part (i) of the theorem, we are done by Remark~\ref{rm}.
\qed

\section{ Concluding remarks }
In the present paper, we investigated cover-free families of finite vector spaces. To establish an upper bound of the cardinality of a cover-free family, we first simplified the problem to the case of a cover-free family of
subspaces with even dimension. Then we defined a family $\mathcal{A}(F, T)$ of $t$-dimensional subspaces with special properties and made use of a weight function to get a tight upper bound of the cover-free family $\mathcal{F}\subseteq \left[V\atop 2t\right]$. Finally, we succeeded in characterizing the structures of all maximum cover-free families, which are intimately connected to $q$-Steiner systems. Our main result is Theorem~\ref{vsc}, which is a generalization of Theorem~\ref{sc} from sets to vector spaces.

It is well-known that a $q$-Steiner system is a type of constant-dimension subspace code with the largest possible size. The present paper disclosed an intimate relationship between $q$-Steiner systems and maximum cover-free families of vector spaces. We think that it is worthwhile to explore possible applications of cover-free families of vector spaces in the field of subspace codes. On the other hand, since the knowledge of $q$-Steiner systems is very limited, the investigation on cover-free families of vector spaces with smaller cardinalities is also interesting.




\begin{thebibliography}{99}

\bibitem{shadow} A. Chowdhury and B. Patk\'{o}s, Shadows and intersections in vector spaces, J. Combin. Theory Ser. A 117(2010), 1095-1106.

\bibitem{a1} B. Chor, A. Fiat, M. Naor and B. Pinkas, Tracing traitors, IEEE Tran. Inform. Theory. 46(2000), 893-910.


\bibitem{chen} W.Y.C. Chen and G.C. Rota, $q$-Analogs of the inclusion-exclusion principle and permutations with restricted position, Discrete Math. 104(1992), 7-22.

\bibitem{b1} M. Dyer, T. Fenner, A. Frieze and A. Thomason, On key storage in secure networks,
J. Cryptology. 8(1995), 189-200.



\bibitem{two} P. Erd\H{o}s, P. Frankl and Z. F\"{u}redi, Families of finite sets in which no set is covered by the union of two others, J. Combin. Theory Ser. A 33(1982), 158-166.

\bibitem{e} P. Erd\H{o}s, P. Frankl and Z. F\"{u}redi, Families of finite sets in which no set is covered by the union of $r$ others, Israel J. Math. 51(1985), 79-89.

\bibitem{ekr vector} P. Frankl and R.M. Wilson, The Erd\H{o}s-Ko-Rado theorem for vector spaces, J. Combin. Theory Ser. A 43(1986), 228-236.

\bibitem{katona} P. Frankl and N. Tokushige, The Katona theorem for vector spaces, J. Combin. Theory Ser. A 120(2013), 1578-1589.

\bibitem{group} F.K. Hwang and V.T. S\'{o}s, Non-adaptive hypergeometric group testing, Studia Sci. Math. Hungar. 22(1987), 257-263.



\bibitem{f} W.H. Kautz and R.C. Singleton, Nonrandom binary superimposed codes, IEEE Tran.
Inform. Theory. 10(1964), 363-377.

\bibitem{Lo} L. Lov\'{a}sz, Combinatorial Problems and Exercises, North-Holland, Amsterdam, 1993.


\bibitem{b2} C.J. Mitchell and F.C. Piper, Key storage in secure networks, Discrete Appl. Math. 21(1988), 215-228.

\bibitem{s}D.R. Stinson and R. Wei, Generalized cover-free families, Discrete Math. 279(2004), 463-477.


\bibitem{a2} J.N. Staddon, D.R. Stinson and R. Wei, Combinatorial properties of frameproof and traceability codes, IEEE Tran. Inform. Theory. 47(2001), 1042-1049.

\bibitem{fra1} C. Shangguan and G. Ge, New bounds on the number of tests for disjunct matrices, IEEE Trans. Inform. Theory. 62(12)(2016), 7518-7521.

\bibitem{c3} R. Safavi-Naini and H. Wang, New results on multi-receiver authentication codes, Advances in Cryptology: Eurocrypt'98, LNCS, 1403(1998), 527-541.


\bibitem{ekrbc} H. Tanaka, Classification of subsets with minimal width and dual width in Grassmann, bilinear forms and dual polar graphs, J. Combin. Theory Ser. A 113(2006), 903-910.
\bibitem{fra} Y. Zhao and X. Zhang, Improved upper bounds for wide-sense frameproof codes, IEEE Tran. Inform. Theory, to appear, DOI:10.1109/TIT.2024.3411150.














\end{thebibliography}
\end{document}